\documentclass[11 pt,a4paper]{amsart}
\usepackage[utf8]{inputenc}
\usepackage{graphicx,color}
\usepackage{latexsym}
\usepackage{graphicx} 
\usepackage{amsthm,amsmath, amssymb,amsfonts}
\usepackage{layout,textcomp}
\usepackage[english]{babel}
\usepackage{cite}

\newcommand\raisepunct[1]{\,\mathpunct{\raisebox{0.5ex}{#1}}}

\theoremstyle{plain}

\newtheorem{theorem}{Theorem}[section]

\newtheorem{lemma}[theorem]{Lemma}

\newtheorem{corollary}[theorem]{Corollary}

\newtheorem{proposition}[theorem]{Proposition}

\theoremstyle{remark}
\newtheorem{remark}[theorem]{Remark}

\numberwithin{equation}{section}
\newcommand{\dive}{{\rm div\,}}

\title{Half-space theorems for $1$-surfaces of $\mathbb{H}^3$}
\author{G. Pacelli Bessa}
\address{Universidade Federal do Ceará\\
Departamento de Matemática\\ 60455-760,\linebreak
Fortaleza - CE, Brazil}
\email{bessa@mat.ufc.br.}
\author{Tiarlos Cruz}
\address{Universidade Federal de Alagoas\\
Instituto de Matemática\\ 57072-970,\linebreak
Maceió - AL, Brazil}
\email{cicero.cruz@im.ufal.br}
\author{Leandro F. Pessoa}
\address{Current: Universit\"at Bielefeld\\ Faculty of Mathematics 33615, Bielefeld, Germany.\linebreak
Permanent: Universidade Federal do Piauí\\
Departamento de Matemática\\ 64049-550, Teresina - PI, Brazil.}
\email{leandropessoa@ufpi.edu.br}

\begin{document}

\begin{abstract}
In this paper we investigate the intersection problem for $1$-surfaces immersed in a complete Riemannian three-manifold $P$ with Ricci curvature bounded from below by $-2$. We first prove a Frankel's type theorem for $1$-surfaces with bounded curvature immersed in $P$ when $\text{\rm Ric}_{P} > -2$. In this setting we also give a criterion for deciding whether a complete $1$-surface is proper. A splitting result is established when the distance between the $1$-surfaces is realized, even if $\text{\rm Ric}_{P} \geq -2$. In the hyperbolic space $\mathbb{H}^3$ we show strong half-space theorems for the classes of complete $1$-surfaces with bounded curvature, parabolic $1$-surfaces, and stochastically complete $H$-surfaces with $H<1$. As a by-product of our techniques a Maximum Principle at Infinity is given for $1$-surfaces in $\mathbb{H}^3_{\raisepunct{.}}$ 
\end{abstract}

\maketitle

\section{Introduction}
An intriguing question in Differential Geometry is whether two minimal surfaces in a Riemannian three-manifold intersect. The roots of this problem can be traced back to  the fifth Euclid's postulate  and its negation leading to discovery of  non-Euclidean geometries. In \cite{hadamard} J. Hadamard showed that on a complete surface with positive curvature
 every geodesic must  intersect every closed geodesic.  T. Frankel in \cite{frankel}  extended the intersection result of Hadamard showing that  minimal hypersurfaces immersed as closed subsets  of a  Riemannian manifold with positive Ricci curvature intersect provided  one of then is compact, see also \cite{petersen_whilhelm}. 
  A class of  minimal hypersurfaces  of a Riemannian manifold is said to have the intersection property if any two elements of   the class intersect unless they are totally geodesic parallel leaves in local product structure.

It was proved  by G.P. Bessa, L.P. Jorge and G. Oliveira in \cite{bessa-jorge-oliveira} that  the class of complete minimal surfaces with bounded curvature immersed  in   three-manifolds $N$ with positive Ricci curvature and bounded geometry  has the intersection property while H. Rosenberg  \cite{rosenberg-bounded} proved it for the class of complete minimal surfaces immersed with bounded curvature in compact three-manifolds $N$ with  positive Ricci curvature.

 In the Euclidean space $\mathbb{R}^{3}_{\raisepunct{,}}$ as a consequence of the  convex hull theorem due to F. Xavier \cite{xavier},   any complete minimal surface  with bounded curvature immersed in an open half-space is a plane. Likewise, in \cite{hoffman-meeks}  D. Hoffman and W. Meeks  proved that any complete  minimal surface  properly  immersed in an open half-space is  also a plane. Furthermore, using a separating plane theorem \cite{meeks-simon-yau}, they proved the intersection property for complete properly immersed minimal surfaces of $\mathbb{R}^3$. In the same vein, Bessa, Jorge, Oliveira \cite{bessa-jorge-oliveira} and Rosenberg \cite{rosenberg-bounded} established a separating plane theorem for  minimal surfaces immersed with bounded curvature, and thus by Xavier's result it yields the intersection property for the class of complete minimal surfaces of $\mathbb{R}^3$ with bounded curvature. These intersection results are known  in the literature as {\em half-space theorems}. 
 
 Half-space theorems in $\mathbb{R}^{3}$ have been   established between the classes of  complete properly  minimal surfaces
and complete minimal surfaces with bounded curvature in \cite[Cor.1.4]{bessa-jorge-oliveira}, as well as between the classes of  parabolic minimal surfaces and complete minimal surfaces with bounded curvature in \cite[Thm.1]{bessa-jorge-pessoa}. 
There are also intersection results for minimal surfaces immersed  in homogeneous three-spaces in \cite{daniel2009half,daniel2011half} and   for surfaces with constant mean curvature $H>0$ immersed in various ambient spaces,  see \cite{M,ros2010properly,RR,rosenberg2013half} and references therein. 

The purpose of this paper is to extend some of this circle of ideas about intersection properties  to  $1$-surfaces immersed in a complete Riemannian three-manifold $P$ with Ricci curvature bounded from below ${\rm Ric}_{P}\geq -2$, and  in particular, $1$-surfaces immersed in the hyperbolic space $\mathbb{H}^3_{\raisepunct{.}}$  

One of our motivation is the pioneering work of R. Bryant \cite{Br}, in which it is shown that the geometry of  minimal surfaces immersed in $\mathbb{R}^3$ shares many similarities with the geometry of $1$-surfaces immersed in  $\mathbb{H}^3_{\raisepunct{.}}$
 This connection has been exploited in several works to provide important contributions to the theory of $1$-surfaces of $\mathbb{H}^3$, see for instance \cite{UY1,CHR,KKMS}.

Let $N$ be a complete $H$-surface properly immersed in a complete oriented three-manifold $P$.  Let $\nu$ be the unit normal vector field along $N$ such that $\overrightarrow{H}_{\!\!_N} = H\nu$,  $H> 0$. A  connected component $\Omega$ of $P\backslash N$ is said to be mean convex if the mean curvature vector field of $N$ points towards $\Omega$. The vector  $\overrightarrow{H}_{\!\!_N}$ at $p_0 \in N \cap \partial \Omega$ points towards $\Omega$ if, for any sequence $q_n \in \Omega$ with $q_n \rightarrow p \in V \subset N$, $V$ a neighbourhood  of $p_0$,  we have $q_n = \exp_p(t_n\nu(p))$ for some $0<t_n<\varepsilon$.

In the context of surfaces with positive constant mean curvature, the intersection property means that an immersed
 $H$-surface can not lie in any mean convex component determined by another disjoint $H$-surface. In our  first result
 we establish  the intersection property for complete $1$-surfaces with
  bounded  curvature immersed in three-manifolds with Ricci curvature ${\rm Ric}>-2$. This result correspond to  the intersection property proved in  \cite{bessa-jorge-oliveira, frankel, rosenberg-bounded} for minimal surfaces with bounded curvature in three-manifolds with ${\rm Ric}>0$.

\begin{theorem}\label{intersect}Let $P$ be a complete Riemmanian three-manifold with Ricci curvature bounded below by ${\rm Ric}_{_P}>-2$, and let $M$ and $N$ be two complete immersed $1$-surfaces  of
 $P$. If $M$ has bounded curvature and $N$ is compact, then $M$ can not lie in a mean convex  component of $P\backslash N$.
\end{theorem}

\begin{remark}
The Ricci curvature assumption in Theorem \ref{intersect} is essential. Indeed, consider the  manifold $P=\mathbb{R} \times \mathbb{T}^{2}$ endowed with the metric $d t^{2}+e^{2 t} g$, where $g$ is the standard flat metric of the torus $\mathbb{T}^{2}_{\raisepunct{.}} $ The manifold $P$  has constant sectional  curvature $K=-1$ thus its Ricci curvature is $-2$ and the slices  $\mathcal{N}_t = \{t\}\times\mathbb{T}^2$ are  compact  $1$-surfaces  embedded in $P$.  Therefore, the slice $M=\mathcal{N}_{t}$ lies in the mean convex component of  $P\setminus \mathcal{N}_{s}$ if $t<s$.
\end{remark}

The proof of Theorem \ref{intersect} relies on the following  stability argument for $1$-surfaces, which is a version of \cite[Cor.3]{S} and can be established using  ideas contained in the proof of \cite[Thm.2.13]{MPR}.

\begin{proposition}\label{MPR-a}
There are no complete strongly stable $H$-surfaces with $H\geq 1$ in a three-manifold with $ Ric>-2$.
\end{proposition}

As an application of Proposition \ref{MPR-a} we have the following counterpart of \cite[Cor.1.6]{bessa-jorge-oliveira} and \cite{rosenberg-bounded} to $1$-surfaces of three-manifolds with Ricci curvature $Ric>-2$.  In what follows a manifold  is said to have  bounded geometry if the sectional curvature is  bounded from above and the injectivity radius is  bounded away from zero. 
 \begin{theorem}\label{teo2}
  Let  $P$ be a  three-manifold  with bounded geometry and Ricci curvature $Ric_{_P} > -2$, and   $M$ be a complete $1$-surface  with bounded curvature injectively immersed of $P$.  Then,
  \begin{itemize} 
  \item[a)] $M$ is compact if $P$ is compact.
  \item[b)] $M$ is proper if $P$ is non-compact.
  \end{itemize}
\end{theorem}
 \begin{remark}Recently W. Meeks and A. Ramos \cite{meeks2019properly} proved that   complete immersed surfaces of finite topology with mean curvature bounded above in a hyperbolic three-manifold  $N$ with sectional curvatures $K_{N}\leq -a^2\leq 0$ under certain assumptions on the  injectivity radius along the ends are proper.  \end{remark}

The hypothesis on the Ricci curvature  in Theorem \ref{intersect} can be relaxed to $Ric_{_P}\geq -2$ as well as  the curvature assumption of $M$ and  the compactness of $N$ if one assumes the existence of a minimizing geodesic realizing the distance ${\rm dist}(M,N)$ and yet yielding a stronger statement, see   Theorem \ref{splitting} below. It   can be viewed as the analogous for $1$-surfaces of \cite[Thm.3.1]{galloway1991intersections}.

\begin{theorem}\label{splitting}
Let $P$ be a complete three-manifold with  $Ric_P\geq -2$. Let $ M$ and $N$ be complete  immersed $1$-surfaces of $P$  that do not intersect and    the distance ${\rm dist}(M,N)$ is realized. If $N$ is proper and $M$ lies in a mean convex component of $P\backslash N$, then 
\begin{itemize}
\item[a)] $M$ and $N$ are embedded totally umbilical equidistant $1$-surfaces.
\item[b)] $M$ and $N$ bound an open connected region in $P$ whose closure is isometric to $[0, l] \times N$, endowed with the metric $dt^2+e^{2t}g$, where $g$ denotes the metric of $N$. If $M$ and $N$ are compact, then  each of them is separating, and the  mean convex component of $P\backslash N$ is isometric to $[0,+\infty) \times N$ with the same metric as before. In particular, $P$ can not be compact.
\end{itemize}
\end{theorem}

The second  goal of this work is to prove versions of the strong half-space theorem for  complete $1$-surfaces immersed in hyperbolic space $\mathbb{H}^3$.  We start with a version for $1$-surfaces immersed in  $\mathbb{H}^3$ of the strong half-space theorem \cite[Thm.1.4]{bessa-jorge-oliveira} between the classes of complete minimal surfaces  with bounded curvature and complete minimal surfaces properly immersed in $\mathbb{R}^{3}$.

\begin{theorem}\label{half_bc}
Let $M$ be a complete $1$-surface immersed in $\mathbb{H}^3$ with bounded curvature and let $N$ be a complete $1$-surface properly immersed in $\mathbb{H}^3$. If $N$ is non-horospherical, then $M$ can not lie in any mean convex component of $\mathbb{H}^3\backslash N$.
\end{theorem}

\begin{remark}  It should be noticed that if $N$ is a non-horospherical properly embedded  $1$-surface of $\mathbb{H}^3_{\raisepunct{,}}$ then each annular end is asymptotic to a catenoid cousin \cite{CHR}. This property contrast from the case of  minimal surfaces properly embedded in $\mathbb{R}^3_{\raisepunct{,}}$ where an annular end can be asymptotic to a  planar  or a catenoid end.
\end{remark}

 The Theorem \ref{half_bc} yields   a version  of the beautiful  Catenoid theorem due to Rodrigues and Rosenberg \cite{RR}, improved by Mazet in \cite{M}, to immersed $1$-surfaces with bounded curvature.

\begin{corollary}\label{corollary_catenoid}
Let $M$ be a complete $1$-surface immersed in $\mathbb{H}^3$ with bounded curvature, and let $C$ be a Catenoid cousin of $\mathbb{H}^3_{\raisepunct{.}}$ Then, $M$ can not lie on the mean convex side of $\mathbb{H}^3 \backslash C$.
\end{corollary}

In \cite[Thm.1.1]{bessa-jorge-pessoa} the authors proved a new version of the strong half-space theorem  between the classes of complete  minimal surfaces  with bounded curvature and of parabolic  minimal surfaces of $\mathbb{R}^{3}$. Recall that a manifold is said to be parabolic (recurrent) if the standard Brownian motion visits any open set at arbitrary large moments of time with probability one, and it is transient otherwise. The simplest examples of  parabolic $1$-surfaces in $\mathbb{H}^3$ are  the immersions conformally to $\mathbb{C}$ or $\mathbb{C}\backslash\{0\}$, for instance horospheres, Enneper and Catenoid cousins (see \cite[Sec.9.3]{GG} for other examples). Our next contribution is a version of \cite[Thm.1.1]{bessa-jorge-pessoa} for $1$-surfaces immersed in $\mathbb{H}^3_{\raisepunct{.}}$

\begin{theorem}\label{parabolic_half}
Let $M$ be a parabolic $1$-surface immersed in $\mathbb{H}^3$ and $N$ be a complete $1$-surface properly immersed in $\mathbb{H}^3$ with bounded curvature. Then, $M$ can not lie in a mean convex  component of $\mathbb{H}^3\backslash N$, unless they are parallel $1$-surfaces.\footnote{In this case $N$ is a horosphere and $M$ could be a horosphere minus a set of zero capacity.}
\end{theorem}

\begin{remark}\label{rmk_mari}
   E. Gama, J. Lira, L. Mari and A. Medeiros in \cite{GLM},   generalizing  results from \cite{RR, ros2010properly}  proved  
  a  theorem related  to Theorem \ref{parabolic_half}  in the case of parabolic surface immersed  into a region $\Omega$ whose boundary $\partial \Omega$ has bounded bending from outwards, remarkably including the case of smooth properly embedded $1$-surface.
\end{remark}

A careful analysis of the proof of Theorem \ref{parabolic_half} shows that we can extend it for surfaces with variable mean curvature provided $\sup H_{M} \geq  \inf H_{N}$. This last inequality is sufficient to apply the Liouville theorem for bounded subharmonic functions which is  equivalent  to parabolicity \cite[Thm.5.1]{grigoryan}. On the other hand, the strict inequality $\sup H_{M} >  \inf H_{N}$ allows us to prove a version of Theorem \ref{parabolic_half} assuming that $M$ is only stochastically complete, compare with \cite[Thm.1.7]{bessa-jorge-pessoa}. 

A Riemannian manifold $M$ is said to be {\it stochastically complete} if for some $(x, t)\in M\times (0, +\infty)$ it holds that $\int_Mp(x,y,t)dy=1$, where $p(x, y, t)$ is the heat kernel of the Laplace operator. Stochastic completeness is equivalent for the following Liouville property: for all $\lambda>0$, any bounded, non-negative solution of the subequation $\Delta u \geq \lambda u$ is identically zero. In particular, every parabolic manifold is stochastically complete. In the following theorem we also give a version of \cite[Thm.1.9]{bessa-jorge-pessoa} for surfaces immersed in  $\mathbb{H}^{3}_{\raisepunct{.}}$

\begin{theorem}\label{stochastic_half-1-surfaces}
Let $N$ be a complete surface properly immersed in $\mathbb{H}^3$ with bounded curvature and let $M$ be a stochastically complete surface of $\mathbb{H}^3$. 
\begin{itemize}
\item[i)] If $\sup H_{M} < \inf H_{N}$, then $M$ can not lie in a mean convex component of $\mathbb{H}^3\backslash N$.
\item[ii)] If $\sup H_{M} = \inf H_{N} > 1$, then $d(M,N) = 0$.
\end{itemize}
\end{theorem}

It is well-understood in the literature that strong half-space theorems give rise to Maximum Principle at Infinity involving surfaces with non-empty boundary. It can be viewed as a generalization of Hopf's Maximum Principle for surfaces with constant mean curvature and it has been investigated in several works, see \cite{langevin-rosenberg,meeks-rosenberg-mp,soret-mp} for minimal surfaces, \cite{lima,meeks-lima} for $H$-surfaces, and \cite{GLM,meeks-rosenberg,ros2010properly} for further generalizations.

 In \cite[Thm.4.2]{meeks-lima}, under a suitable hypothesis of ideal contact at infinity, it was established a Maximum Principle at Infinity for proper surfaces of $\mathbb{H}^3$ with bounded mean curvature, but not both equal to $1$. In \cite[Thm.1]{GLM} it was  proved  a Maximum Principle at Infinity for parabolic $1$-surfaces with boundary immersed into a region $\Omega$ of a  Riemannian manifold $P$ with Ricci curvature bounded ${\rm Ric}\geq -2$, and whose $\partial \Omega$ has bounded curvature and bounded bending from outwards (see Remark \ref{rmk_mari}). Recall that a surface $M$ with non-empty boundary $\partial M$ is said to be parabolic if the absorbed Brownian motion is recurrent, that is, any Brownian path starting from an interior point of $M$, reaches the boundary (and dies) in a finite time with probability $1$ (see \cite{perez-lopez}). From a potential-theoretic viewpoint \cite[Prop.10]{pessoa-pigola-setti}, the parabolicity is equivalent to the following Ahlfors maximum principle: every weak bounded solution $u \in C^0(M)\cap W^{1,2}_{\text{loc}}({\rm int} M)$ of the subequation $\triangle u \geq 0$ in ${\rm int} M$ must satisfies
\begin{eqnarray*}
\sup_{M} u = \sup_{\partial M} u.
\end{eqnarray*}
It should be remarked that the usual definition of parabolicity for surfaces with boundary for which the Brownian motion reflects at $\partial M$ is stronger than the above notion (see \cite{impera-pigola-setti,pessoa-pigola-setti}).

In our last result we provide the hyperbolic version of the Maximum Principle at Infinity proved in \cite[Thm.1.11]{bessa-jorge-pessoa} for parabolic $1$-surfaces.

\begin{theorem}\label{maximum-principle-infinity}
Let $M$ and $N$ be disjoint immersed surfaces of $\mathbb{H}^3_{\raisepunct{.}}$ Assume $M$ is parabolic with non-empty boundary $\partial M,$ and $N$ is a complete surface properly immersed with bounded curvature. If $\sup_{\!_M}\vert H_{\!_M}\vert \leq \inf_{\!_N}H_{\!_N}>0$ and $M$ lies in a mean convex component of $\mathbb{H}^3\backslash N$, then
\[
{\rm dist}(M,N) = {\rm dist}(\partial M,N).
\]
\end{theorem}

\noindent \textbf{Acknowledgements.} This work was partially supported by Alexander von Humboldt Foundation and Capes-Brazil (Finance Code 001), and by CNPq-Brazil, Grants 303057/2018-1, 311803/2019-9 and 306738/2019-8. The third author is grateful to Professor Alexander Grigor'yan and the Faculty of Mathematics at the Universit\"at Bielefeld for their warm hospitality.

\section{Strong Stability of $H$-surfaces}\label{preli}

Let $P$ be a Riemannian three-manifold and let $\phi \colon M\to P$ be a surface isometrically immersed in $P$. Let $\Phi\colon (-\epsilon,\epsilon)\times M \to P$ be  a  variation of $M$ with $\Phi_t(p)=\Phi(t,p)$ and $\Phi (0,p) = \phi(p)$, where each $\Phi_t$ is an immersion of $M$ into $P$ for every $0< \vert t\vert <\epsilon$. For each $t\in (-\epsilon, \epsilon) $ we have the area function $A(t) = {\rm Area}(\Phi_t)$ and the volume function $V(t)$ induced by the immersion $\Phi$   given by
\[V(t)=\int_{[0,t]\times M}\Phi^*dV, \,\footnote{We  agree that $[0,t]=[t,0]$ if $t<0$} \] which measures the signed volume enclosed between $\Phi_0=\phi$ and $\Phi_t$.
Let us define the functional  $\mathcal J$ setting $\mathcal J(t) =  A(t) - 2H V(t)$.
The variational  vector field $X$ associated to $\Phi$ is defined by
$X=  {\partial_t \Phi}_{|_{t=0}} = \psi \nu$, for some $\psi \in C^{\infty}(M)$.
   It is not difficult to check that  $M$ is a stationary point for $\mathcal{J}$ if and only if it has constant mean curvature $H$.  

If we assume that $M$ is stationary, then the second variation formula of $\mathcal J(t)$ is given by
\begin{eqnarray*}
Q(\psi,\psi) &=&  -\int_{M}\psi L\psi\;d\sigma \\[0.2cm] 
&=& \int_{M} [|\nabla\psi|^2-(|II|^2+Ric_{_P}(\nu,\nu))\psi^2]\;d\sigma  \quad \forall \psi \in C^{\infty}_0(M),
\end{eqnarray*}
where $Ric_{_P}$ is the Ricci curvature of $P$,  $II$ is the second fundamental form of $M$ and  $L=\Delta+|II|^2+Ric_{_P}(\nu,\nu)$ is its Jacobi operator. An $H$-surface $M$ is said to be strongly stable if $Q(\psi,\psi)\geq 0$. This notion is  equivalent to the positiveness of the first eigenvalue of $L$ and to the existence of a positive smooth solution $u$ for the equation $L u = 0$ (see \cite{FC}). For $H$-surfaces there is also a weaker notion of stability associated to the isoperimetric problem, that is for minimizing the area of $M$ while keeping enclosed a constant volume. An $H$-surface is stable if $Q(\psi,\psi) \geq 0$ for every test $\psi \in C^{\infty}_0(M)$ satisfying $\int_{M}\psi d\sigma = 0$. Hence, strong stability implies stability, but not otherwise.



In this section we are interested in to study the strong stability of leaves from the limit set of surfaces  with bounded curvature. Let $\varphi \colon M \to P$ be a complete surface immersed into a complete three-manifold $P$. The limit set of $\varphi$, denoted by $\mathcal{L}_{\varphi}$, is the set 
$$
\mathcal{L}_{\varphi}=\{q\in P\colon \exists \{p_k\}\subset M,\mbox{dist}_M(p_0,p_k)\to\infty\mbox{ and } \mbox{dist}_P(q,\varphi(p_k))\to 0\}.
$$
It is plain to see that if $M$ is properly immersed, then $\mathcal{L}_{\varphi}=\emptyset$.

An important tool in our study is the maximum principle for $H$-surfaces. Suppose $M_1$ and $M_2$ are two smooth oriented surfaces of $P$ which are tangent at a point  $p\in M_1\cap M_2$ and have at $p$ the same oriented normal $\nu$. The point $p$ is called a point of common tangency. Around $p$ let us express $M_1$ and $M_2$ as graphs of functions $u_1$ and $u_2$ over the common tangent plane through $p$. We shall say that $M_1$ lies above $M_2$ near $p$, if $u_1 \geq u_2$ in a neighborhood of $p$. We can now state the following \textit{maximum principle} for $H$-surfaces (c.f. \cite{coskunuzer}).

\begin{lemma}\label{maximum-principle}
Let $M_1$ and $M_2$ be oriented surfaces immersed in a complete three-manifold $P$. Assume $M_1$ and $M_2$ have a point of common tangency $p$ and let $H_1$ and $H_2$ be their respective mean curvature functions with respect to the same normal. If $H_1 \leq H_2$ near $p$, then $M_1$ can not lie above $M_2$, unless $M_1$ coincides with $M_2$ in a neighborhood of $p$. 
\end{lemma}

Given $M_1$ and $M_2$  two  $1$-surfaces immersed   in a complete three-manifold $P_{\raisepunct{.}}$ A point $p\in M_1\cap M_2$ of   common tangency is said to be a kissing point if $M_1$ lies above $M_2$ but they do not coincide in a neighborhood of $p$, that is, the mean curvature vectors of $M_1$ and $M_2$ at $p$ point to opposite sides. Unlike the minimal case, an $1$-surface can have a tangential self-intersection at a kissing point $p$. Such a point is called a self-touching point. 
This means that the $1$-surface is immersed but do not cross itself.


In the following result we generalize \cite[Thm.1.5]{bessa-jorge-oliveira} and, although  we state it for dimension three, it holds for any dimension.

\begin{theorem} \label{Theobessa}
Let $P$ be a complete three-manifold with bounded geometry and Ricci curvature $Ric_{_P}\geq -2$. Let  $\varphi \colon M \to P$ be  a complete $1$-surface immersed in $P$ with Gaussian curvature bounded from below. Then, one of the following conditions holds.
\begin{itemize}
\item[(a)]$\varphi$ is proper;
\item[(b)] Every complete leaf $S\subset \mathcal{L}_{\varphi}$ whose intersection with $\varphi(M)$ is either empty or only admits kissing points is strongly stable. 
\end{itemize}
\end{theorem}
\begin{proof}

Suppose $\varphi$ is a non-proper immersion and let $p\in \mathcal{L}_{\varphi}$. 
\vspace{0.2cm}

\noindent\textbf{Claim 1:}  There exists a sequence of  disks $\{D_k\}$ in $P$ converging uniformly to a disk $D\subset \mathcal{L}_{\varphi}$ containing $p$. 
Moreover, the disk $D$ can be extended to a complete $1$-surface $S\subset \mathcal{L}_{\varphi}$ passing through $p$ with bounded curvature and $H=1$.

For the sake of completeness we will briefly outline a proof for this claim. Arguing under the non-properness of $\varphi$ we can take a divergent sequence $x_k \in M$ such that $p_k = \varphi(x_k)$ converges to the point $p \in \mathcal{L}_{\varphi}$. 
 Since  $\varphi(M)$ is an $1$-surface with bounded curvature, and $P$ has bounded geometry there is a uniform bound on the second fundamental form of $\varphi(M)$. Therefore, there exists a family of disks $D_k(\delta) \subset T_{p_k} \varphi(M)$, centered at the origin and with uniform radius $\delta > 0$, such that $\varphi(M)$ is locally described as the graph of a function $u_k$ which enjoy a $C^1$ bound independently of $p_k \in \varphi(M)$. We pick a subsequence of $p_k$, still called $p_k$, such that $T_{p_k}\varphi(M)$ converges to a vector subspace $V \subset T_p P$, determined by a fixed unit normal vector field $\nu$, with the property that the sign between the mean curvature vector field $\overrightarrow{H}(x_k)$ and $\nu$ is fixed. For $k$ sufficiently large, the local graphs over $D_k(\delta)$ are also graphs on a small disk $D(\delta/2) \subset V$. The classical quasilinear PDE theory asserts that these graphs converges to a limit $1$-graph $S$ tangent to $V$ at $p$. Since all boundary points of $S$ are in $\mathcal{L}_{\varphi}$, reasoning as above, we can extend $S$ to a geodesically complete, oriented leaf contained in $\mathcal{L}_{\varphi}$ with bounded curvature, also denoted by $S$ (see also \cite{Ronaldo}). 

To prove assertion $(b)$ we argue along similar lines from \cite[Thm.1.5]{bessa-jorge-oliveira}. The argument is inspired by \cite[Thm.1]{R}. 
Let $S\subset \mathcal{L}_{\varphi}$ be the complete $1$-surface with bounded curvature passing through $p$ constructed in {\bf Claim 1}. Since the intersection $S\cap \varphi(M)$ is either empty or only admits kissing points, then $S$ has no transversal self-intersection. Moreover, at possible tangential self-intersection points the maximum principle (Lemma \ref{maximum-principle}) implies that the mean curvature vector field along $S$ must point in opposite directions, thus $S$ admits only self-touching points.

Let $C \subset S$ be a compact proper subset of $S$ and let $T^+_{\varepsilon}(C)$ be the oriented $\varepsilon$-tubular neighborhood of $C$ in $P$, with respect to the mean curvature vector field of $S$.  For some  $\varepsilon>0$, depending on the curvature bounds of $M$ and $P$, the  $\varepsilon$-tube  $T^+_{\varepsilon}(C)$ is embedded. Consider a sequence of compact subsets $C_k \subset \varphi(M)$ converging uniformly to $C$. From our assumption on $\varphi(M)\cap S$, even in the case that $C$ contains a kissing point between $\varphi(M)$ and $S$, or a self-touching point of $S$, we can guarantee that, up to a subsequence, $C_k$ converges to $C$ on one side of $C$, that is, inside $T^+_{\varepsilon}(C)$. Let us denote by $\nu$ be the continuous unit normal vector field along $S$ pointing towards $C_k$. By the construction of $C$, the mean curvature vector fields of $C_k$ point towards the same direction as the mean curvature vector field of $C$, since $C_k$ converges uniformly to $C$ by one side.

\vspace{2mm}

\noindent\textbf{Claim 2:}   $C$ is strongly stable.

To prove the claim, we take a compact $\widetilde C$ containing properly  $C$. If the first eigenvalue of the Jacobi operator $L=\Delta+|II|^2+Ric_{\!_P}(\nu,\nu)$ in $\widetilde C$ is non-negative,  then $\widetilde{C}$ is strongly stable and we are done. Therefore, we can assume that $\lambda_1^{L}(\widetilde{C})< 0$, and in this case, there exists a smooth function $u$ on $\widetilde C$ satisfying
\[
\left\lbrace
\begin{array}{rl}
Lu =1 &  \text { in}\ \ \widetilde C, \\[0.2cm]
u =0 & \text { on}\ \ \partial \widetilde C.
\end{array}\right.
\]

Consider the variation $\widetilde{C}(t)=\{\exp_x(tu(x) \nu) \colon x\in \widetilde C\}$ for $-\varepsilon<t<\varepsilon$, and denote by $H(t)$ its mean curvature function. The mean curvature $H(t)$ evolves, accordingly
to \cite[Thm.3.2]{HP}, as  
$$
 H'(0)=\frac{1}{2}Lu=\frac{1}{2}\,\,\mbox{ in } \widetilde C.
$$
Recalling that $H(0)=1,$ we have $H(t)>1$ for every $t\in(0,\epsilon')$, for some $0<\varepsilon' < \varepsilon$.  Suppose  that $u$ is positive at some  point of $\mbox{int}(\widetilde C)$,  then there is a  small enough $t$ such that $\widetilde C (t)$  have a tangency point with some $C_k$ which is not possible by the maximum principle (Lemma \ref{maximum-principle}). We may conclude that $u \leq 0$. Similarly, if $u\leq 0$ and attains its maximum at some interior point of $\widetilde C$, then we conclude that $u\equiv 0$ which is impossible since $L u = 1$. Therefore, $u < 0$ in $\mbox{int}(\widetilde C)$ and it vanishes on the boundary $\partial \widetilde C$.

Set $w = -u$ and consider $v$ to be a positive first eigenfunction of $C$, that is, $v$ is the solution of the problem 
\[
\left\lbrace
\begin{array}{rl}
Lv +\lambda_{1}^{L}(C) v = 0 & \text { in}\ \ C, \\
v = 0 & \text { on}\  \partial  C.
\end{array}\right.
\]
Define the function $h=w-\tau v$ in $C$. We can choose $\tau$ such that $h\geq 0$ and $h(p)=0$ for some $p\in\mbox{int}(C) $. Suppose by contradiction that $\lambda_{1}^{L}(C)<0$. Then,
$$
Lh=Lw-\tau Lv < \tau \lambda_{1}^{L}(C) v\leq 0 \  \ \mbox{in } C.
$$
Therefore $h$ is superharmonic on $C$ and has a minimum ($h(p)=0)$ at  $\mbox{int}(C)$. By the maximum principle $h$ is constant, a contradiction. Hence any $C\subset S$ is strongly stable and thus $S$ is strongly stable. 
\end{proof}

Keeping the  hypotheses of Theorem \ref{Theobessa} we have the following corollary.
\begin{corollary}\label{corbessa}
 If $S$ is a compact leaf of $\mathcal{L}_{\varphi}$, then $S$ is totally umbilical and  $Ric_{P}(\nu,\nu)=-2.$
\end{corollary}
\begin{proof}Since $S$ is strongly stable,
by a result of D. Fisher-Colbrie \cite{FC}, there exists a positive solution $u$ of the Jacobi operator $Lu =0$. Integrating over $S$, we then obtain 
\begin{eqnarray*}
\int_S (|II|^2+Ric_{P}(\nu,\nu))u=0.
\end{eqnarray*}
 Since $Ric_{\!_P}\geq -2$ and $|II|^2\geq 2H^2=2,$ we conclude that $S$ is totally umbilical and  $Ric_{P}(\nu,\nu)=-2.$
\end{proof}

 We are now going to prove Proposition \ref{MPR-a} stated in the Introduction.

\begin{proposition}\label{MPR}
There are no complete  strongly stable $H$-surfaces with $H\geq 1$ immersed in a three-manifold with Ricci curvature $Ric_{\!_P}> -2$.
\end{proposition}
\begin{proof}
Suppose by contradiction that there exists a complete strongly stable $H$-surface $M$ with $H\geq 1$. We  assume that $M$ is non-compact, otherwise the constant function $1$  in the stability inequality will give a contradiction. Let $x_0\in M$ and $R>0$. It follows from the proof of item $4$ in \cite[Thm.2.13]{MPR} that we can find a constant $C>0$ such that
\begin{equation}\label{stabmpr}
0<\int_{M}(|II|^2+Ric_{P}(\nu,\nu))f^2\leq \frac{C}{\log R}\raisepunct{,}
\end{equation}
where $f(q)=\varphi(r)$ is a radial logarithmic cut-off function given by
\[
\varphi(r)=\left\{
\begin{array}{ccc}1 & \text { if } & 0 \leq r \leq 1, \\[0.2cm] 
\displaystyle 1-\frac{\log r}{\log R} & \text { if } & 1 \leq r \leq R, \\[0.3cm]
0 & \text { if } & R \leq r.\end{array}\right.
\]
Above $r(q)={\rm dist}(x_0,q)$ denotes the intrinsic distance from $q$ to $x_0$. The last right-hand side of \eqref{stabmpr} goes to $0$ as $R$ tends to infinity, while  the integrand is strictly positive. This leads to a contradiction.
\end{proof}

The following result is a straightforward consequence of Proposition \ref{MPR}.

\begin{corollary}
 Let $M$ and $N$ be two disjoint $1$-surfaces properly embedded  in a Riemannian three-manifold $P$ with Ricci curvature ${\rm Ric}_P > -2$. Then, $M$ and $N$ can not bound a mean convex component between them. 
\end{corollary}
\begin{proof}
If  $\Omega$ is a mean convex component whose boundary are $M$ and $N$, then by the proof of \cite[Thm.3]{RR} there is a strongly stable $1$-surface in $\Omega$, which is a contradiction by Proposition \ref{MPR}.  
\end{proof}

\section{Proof of Theorems \ref{intersect}, \ref{teo2} and \ref{splitting}}

In this section we are going to present the proofs of the results that are consequence of Theorem \ref{Theobessa} and Proposition \ref{MPR}. Although it is an  unnatural ordering,  for simplicity  as it will be clear afterwards, we will leave the proof of Theorem \ref{intersect} to the last part of this section.

\subsection{Proof of Theorem \ref{teo2}}

To prove item $a)$ we assume by contradiction that $\varphi\colon M\to P$ is a complete non-compact $1$-surface injectively immersed with bounded Gaussian curvature in $P$. First, we observe that  since  $M$ is non-compact and $P$ is compact $\mathcal{L}_{\varphi}\neq \emptyset$. Moreover,  there exists $\epsilon>0$ depending on the second fundamental form of $M$ and on the bounds of the geometry of $P$ such that $\varphi^{-1}(B_\epsilon^{P}(q))$ is a countable union of disjoint disks $D_i(\delta)\subset M$, with $\delta=\delta(\epsilon)>0$ for every $q\in \mathcal{L}_{\varphi}$, see \cite[Lem.1--3]{jorge-xavier5} and their proofs.

We claim that $\varphi(M)$ can not lie in $\mathcal{L}_\varphi$, otherwise we can pick a point $q\in \varphi(M) \subset \mathcal{L}_\varphi$. As in {\bf Claim 1} of the proof of Theorem \ref{Theobessa} there is a family of disjoint disks $D_j\subset \varphi(M)$ converging to a disk $q\in D_{\infty}\subset \mathcal{L}_\varphi$. Take  a minimizing geodesic $\gamma \colon [0, \epsilon)\to P$ with $\gamma(0)=q$ and $\gamma'(0)\perp T_qD_{\infty}$ pointing towards $D_j$. Since each point $q_j\in \gamma([0, \epsilon))\cap D_j\neq \emptyset$ belongs to $\mathcal{L}_\varphi$, it must be an accumulation point of $ \Gamma=\gamma([0, \epsilon))\cap \varphi (M)$.  Therefore $\Gamma$ is   a perfect set and $\varphi^{-1}(B_\epsilon^{P}(q))$ is uncountable, contradiction.

Now, let $S \subset \mathcal{L}_{\varphi}$ be a complete leaf passing through a point $p \in P$. 
If $S$ would intersect $\varphi(M)$ transversally, then $\varphi $ would be not injectively immersed. Thus $S$ might intersect $\varphi (M)$ tangentially. Furthermore, such a intersection point must be a kissing point, for if the mean curvature vector fields of $S$ and $\varphi(M)$ coincide and by maximum principle $\varphi(M)= S \subset  \mathcal{L}_{\varphi}$, which gives a contradiction.

In any case we can conclude that $S$ might intersect $\varphi(M)$ only at kissing points, and by Theorem \ref{Theobessa} it must be strongly stable. This fact contradicts Proposition \ref{MPR}.

For the proof of item $b)$ we let $P$ be a non-compact manifold and assume by contradiction that $M$ is non-proper, that is $\mathcal{L}_{\varphi}\neq\emptyset$. Arguing as above we obtain a complete leaf $S \subset \mathcal{L}_{\varphi}$ whose intersection with $\varphi(M)$ is either empty or only contains kissing points. Again, $S$ is strongly stable by Theorem \ref{Theobessa} and this leads to a contradiction with Proposition \ref{MPR}.

\subsection{Proof of Theorem \ref{splitting}}

Let $\gamma:[0,l] \rightarrow P$ be a minimizing geodesic realizing the distance between $M$ and $N$  with initial data $x_{1}=\gamma(0) \in N$ and $x_{2}=\gamma(l) \in M$. By the first variation formula of arc-length the geodesic $\gamma$ must intersect $M$ and $N$ orthogonally. We define the normal exponential map $\Phi:[0, l] \times U_{1}\to P$ by $\Phi(t, x)=\exp _{x} (t\nu)$, where $U_{1}$ is a neighborhood of $N$ containing $x_1$ and $\nu$ is the normal vector field coinciding with $\gamma^{\prime}(0)$.

Since $M$ is an $1$-surface and $\gamma(t)=\Phi\left(t, x_{1}\right)$ realizes the distance between $M$ and $N$, the Jacobian of $\Phi$ at $\left(l, x_{1}\right)$ is non-singular. Then, up to shrinking $U_{1}$, we can produce a foliation in a tubular neighborhood of $N$ given by regularly embedded surfaces $V_{t}=\Phi\left(t, U_{1}\right)$. It is now standard that  the unit tangent vector field to the normal geodesic $\nu = \Phi(\partial_t)$ is parallel and  each $V_t$ is equidistant to $U_{1}.$ Recalling that $M$ lies in a mean convex component of $P\backslash N$, the mean curvatures $ H(t)$ of these surfaces satisfy (see Lemma \ref{gray_lemmata})
\begin{eqnarray*}
 2H'(t)&=& Ric_P(\nu,\nu)+|II(t)|^{2},
\end{eqnarray*}
where $II(t)$ is the second fundamental form of $V_{t}$. By Newton's inequality the function $H(t)$ satisfies $H'(t) \geq  H^2(t) - 1$ with $H(0) = 1$. Using the Riccati's comparison theorem we conclude that $H(t) \geq 1$ and thus $H'(t) \geq 0$.


The surface $V_{l}=\Phi(l,U)$ must be tangent to a neighborhood $U_2 \subset M$ containing $x_{2}$. Since $d(U_{1}, V_{l})=l$, we have that  $V_{l}$  stays below $U_{2}$ and  by Lemma  \ref{maximum-principle} they must coincide near $x_2$. This means that $H(t) \equiv 1$. Now, if another piece of $M$ touches $U_{2}$, then they must coincide provided $U_{2}\subset V_{l}$. Observing that the distance from $M$ to $N$ is also realized at the boundary points of $V_l$, we can apply the precedent argument to continue $V_l$ parallel to $N$, showing that $M$ and $N$ are equidistant everywhere, unless they intersect and, in this case, they are equal. Since the foliation $V_t$ is given by $1$-surfaces we have
$$
Ric_P(\nu,\nu)=-2\quad \mbox{and}\quad|II(t)|^{2}=2.
$$
Therefore, $M$ and $N$ are embedded and totally umbilical.

To prove item $b)$ we consider $\Omega$ be the connected component of $P$ whose boundary contains $M$ and $N$. Then, $\Omega$ is foliated by umbilical surfaces. The induced metric on each leaf  evolves as $ g'(t)=2 g(t)$ (see \cite{HP}). Therefore, the metric induced by $\Phi(t, x)$ on $[0,l] \times M$ is given by $d t^{2}+e^{2 t} g$, where $g$ is the metric of $N$. 

Finally, suppose by contradiction that $N$ is compact, but not separating. Then we argue along similar lines from the proof of  \cite[Thm.2.5]{CF} in order to  construct a cyclic cover $\hat P$ of $P$. Since $N$ is two-sided, we can define a smooth function on $P\backslash N$ which is equal to $0$ on $N$ and in a neighborhood of one side of $N$, and  equal to $1$ in a neighborhood of the other side of $N$. By  passing to the quotient $\bmod\; \mathbb{Z}$ we obtain a non-constant smooth function
$$
f: P \rightarrow \mathbb{R} / \mathbb{Z}=\mathbb{S}^{1}.
$$

Let $\tilde{P}$ be the universal cover  of $P$ and $f_{*} \colon \pi_1(P) \to \mathbb{Z}$ be the induced map  on the fundamental groups. Then $\hat{P}=\tilde{P} / \mbox{ker} f_{*}$ is a cyclic cover  of $P$ and  the preimage of $N$ under the projection $\pi: \hat{P} \rightarrow P$ divides $\hat{P}$ into two infinite parts.  Choosing the component with adequate normal direction, say $\Omega$, it follows from \cite[Lemma 1]{schoen1982lower} that $(n-1)\operatorname{Vol}(\Omega) \leq \operatorname{Vol}(\partial \Omega)<\infty,$ which is a contradiction. Thus $N$ is separating. The last assertion follows from the splitting theorem \cite[Thm.2]{croke1992warped}.

\subsection{Proof of Theorem \ref{intersect}}


Let $M$, $N$ be complete $1$-surfaces immersed in a compact three-manifold $P$ with Ricci curvature ${\rm Ric}_{_P}> -2$. Recall that $N$ is proper and $M$ has bounded curvature. Suppose by contradiction that $M$ lies in a mean convex component $\Omega$ of $P\backslash N$. Given a point $p \in \overline{M}$ let $S \subset \mathcal{L}_\varphi$ be the complete $1$-leaf immersed in $P$ with bounded curvature and passing through $p$, given by {\bf Claim 1} in the proof of Theorem \ref{Theobessa}.


\vspace{0.2cm}
\noindent \textbf{Case 1:}  $ \overline{M} \cap N\neq \emptyset$.
\vspace{0.2cm}

Let $p \in \overline{M} \cap N$ and $S \subset \mathcal{L}_\varphi$ be the leaf passing through $p$. Recalling that $S$ is obtained by uniform limit of disks with radius uniformly bounded from below, we may assume that $S$ intersects $N$ tangentially, otherwise there were some disk of $M$ intersecting $N$ transversally, contradicting $M\cap N = \emptyset$. Moreover, by the maximum principle (Lemma \ref{maximum-principle}) $S$ must coincide with $N$ in a neighborhood $U$ of $p$. Thus, repeating the above argument with points of $\partial U$ we can show that $S \subset N$. 

We have proved that $S$ is an $1$-surface, which is a leaf from $\mathcal{L}_\varphi$ satisfying $S\cap \varphi(M) = \emptyset$. Then, by item $(b)$ of Theorem \ref{Theobessa} it must be strongly stable. However, this contradicts Proposition \ref{MPR}.

\vspace{0.2cm}
\noindent\textbf{Case 2:}  $\overline{M}\cap N= \emptyset.$
\vspace{0.2cm}

Since $N$ is compact there exist points $p \in \overline{M}$ and $q \in N$ such that $d(\overline{M}, N)= d(p,q)= l$, for some $l > 0$. We take $S \subset \mathcal{L}_\varphi$ be the $1$-surface immersed with bounded curvature and passing through $p$.  
In this case, there exists a minimizing geodesic $\sigma \colon [0,l] \rightarrow P$ realizing the distance between $N$ and $S$ such that $\sigma(0) = q \in N$ and $\sigma(l) = p \in S$. Since $S$ lies in a mean convex component of $P\backslash N$ by item $a)$ of Theorem \ref{splitting}, $S$ and $N$ are embedded and totally umbilical parallel $1$-surfaces. However, we know that the mean curvatures $H(t)$ of the parallel surfaces of $N$, in the direction of its mean curvature vector field, satisfy (see Lemma \ref{gray_lemmata})
\begin{eqnarray*}
 2H'(t)&=& Ric_P(\nu,\nu)+|II(t)|^{2} \\[0.2cm]
 &>& -2 + 2H^2(t) > 0,
\end{eqnarray*}
where in the last inequality we have used that $H(t) \geq 1$. Therefore, any neighborhood $U \subset N$ of $q$ evolving paralleling along $\gamma$ must intersect $S$ tangentially at $p$ with mean curvature $H(l) > 1$. Once again, we apply the maximum principle to get a contradiction.

\section{Proof of Theorem \ref{half_bc}}

The following lemma is the core of the proof of Theorem \ref{half_bc} and it is the correspondent of \cite[Thm.1.2]{bessa-jorge-oliveira}.

\begin{lemma}\label{horosphere_lim_set}
Let $\Omega\subset\mathbb{H}^3$ be an open domain whose boundary is a union of pieces of regular $1$-surfaces with respect to normal vector field pointing 
towards the interior of $\Omega$. Assume $\varphi \colon M \to \Omega \subset \mathbb{H}^3$ is a complete $1$-surface immersed with bounded curvature.
Then there is a horosphere $\mathcal H$ separating $\overline{M}$ from $\partial \Omega$, unless $\partial \Omega$ is a horosphere contained in the limit set $\mathcal{L}_\varphi$. 
\end{lemma}

\begin{proof}
We first consider the case where $\mathcal{L}_\varphi \cap \partial \Omega \neq \emptyset$. It then follows from Theorem \ref{Theobessa} and the maximum principle that $\partial \Omega$ is a leaf from $\mathcal{L}_\varphi$ which is a complete strongly  stable $1$-surface. However, in this case $\partial \Omega$ is a horosphere by \cite[Thm.2.13]{MPR}. 
  
Let us assume that $\mathcal{L}_\varphi \cap \partial \Omega = \emptyset$. Following ideas from \cite{bessa-jorge-oliveira} we consider an open ball $B_R$ centered at some point of $\mathbb{H}^3$ and radius $R>0$ which intersects both $\partial \Omega$ and $\overline{M}$. Since $\overline M$ lies in $\Omega$ and has bounded curvature, there exits $0<r< (\sup \sqrt{K_M})^{-1}$ (depending on $R$) such that the tubular neighborhood $T_r(\overline{M})$ has Lipschitz boundary (see \cite{RZ}), $T_r(\overline{M})\cap \partial\Omega = \emptyset$ and the outside tangent cone of $\partial T_r(\overline{M})$ has no angle bigger than $\pi$ (cf. \cite[Lemm.2.1]{bessa-jorge-oliveira}). Let $D_R$ be the connected open region of $(\Omega \backslash T_{r/2}(\overline{M}))\cap B_R $ whose boundary contains $\partial \Omega \cap B_R$. 
Precisely, $\partial D_R$ is composed by smooth pieces $\partial B_R \cap \Omega$, and by piecewise $1$-surfaces from $\partial \Omega \cap B_R$ and $\partial T_{r/2}(\overline{M})\cap B_R$.

Let us denote by $\partial T_{r/2}$ be the part of $\partial D_R$ contained in $\partial T_{r/2}(\overline{M})$. Let $\mathcal{F}$ be the class of open domains $Q \subset D_R$ with rectifiable boundary such that $\partial T_{r/2} \subset \partial Q$. We define on $\mathcal{F}$ the functional 
\begin{equation*}
F(Q) = A(\partial Q) - 2V(Q),
\end{equation*}
where $A(\partial Q)$ gives the area of the boundary of $Q$, and $V(Q)$ gives the volume of $Q$. The idea now is to minimize the functional $F$ and work along the strategy inspired in \cite{hauswirth_roitman_rosenberg,ros2010properly} and implemented in \cite[Sec.5]{M}. For the sake of completeness we will write down a sketch of the argument.

Since $\partial \Omega$ is a piecewise regular $1$-surface whose mean curvature vector field points inward $\Omega$, for $\mu>0$ (depending on $R$) sufficiently small, the subset $\Omega_\mu = \{ x \in \Omega \colon d(x,\partial \Omega) \leq \mu \}$ does not intersect $\partial T_r$, and it is foliated by piecewise smooth surfaces whose mean curvature is greater than $1$ where it is defined. Further, the curvatures bounds of $M$ imply that $ T_{r}(\overline{M})$ is also foliated by piecewise smooth surfaces, but in this case the mean curvature evolution along these smooth pieces depends on the direction of the mean curvature vector field of the corresponding limit disks given in Theorem \ref{Theobessa}.

\vspace{0.1cm}
\noindent{\bf Claim:} Let $Q \in \mathcal{F}$.
\begin{itemize}
\item[i)] If $Q \cap \Omega_{\mu/2}\neq \emptyset$, then there is $\mu' \in (\mu/2,\mu]$ such that  $Q\backslash \Omega_{\mu'} \in \mathcal{F}$ and $F(Q\backslash \Omega_{\mu'}) \leq F(Q)$.
\vspace{0.1cm}
\item[ii)] If  $T_{r}(\overline{M}) \nsubseteq Q $, then there is $r' \in (r/2,r]$ and $Q' \in \mathcal{F}$ such that  $T_{r'}(\overline{M}) \subset Q'$ and $F(Q') \leq F(Q)$.
\end{itemize}


\noindent To prove item $i)$ we denote by $\xi$ the inward unit normal vector field induced by the foliation of $\Omega_\mu$. It is easy to see that $\dive \xi \leq -2$ a.e. in $\Omega_\mu$. Using that $\partial Q$ has finite two dimensional Hausdorff measure, we can apply the coarea formula to find $\mu' \in (\mu/2,\mu]$ such that the one dimensional Hausdorff measure of $\partial Q\cap \partial \Omega_{\mu'}$ is finite. Hence, it will be a negligible subset in our computations. Then, the subset $Q\cap \Omega_{\mu'} \neq \emptyset$ has rectifiable boundary and the Stokes formula (see e.g. \cite[Sec.2.2]{M}) gives
\begin{eqnarray*}
-2V(Q\cap \Omega_{\mu'}) &\geq & \int_{Q\cap \Omega_{\mu'}} \dive \xi\\[0.2cm]
&=& \int_{\partial Q\cap \Omega_{\mu'}} \langle \xi, \eta(Q\cap \Omega_{\mu'})\rangle + \int_{\partial\Omega_{\mu'}\cap Q} \langle \xi, \eta(Q\cap \Omega_{\mu'})\rangle,
\end{eqnarray*}
where $\eta(Q\cap \Omega_{\mu'})$ denotes the outward unit normal vector field along the boundary of $Q\cap \Omega_{\mu'}$. Since $\xi = \eta(Q\cap \Omega_{\mu'})$ on the subset $\partial\Omega_{\mu'}\cap Q$, by the Cauchy-Schwarz inequality we have
\begin{eqnarray*}
-A(\partial Q\cap \Omega_{\mu'}) + A(Q\cap \partial\Omega_{\mu'}) + 2V(Q\cap \Omega_{\mu'}) \leq 0.
\end{eqnarray*}
This implies that 
\begin{eqnarray*}
F(Q\backslash \Omega_{\mu'}) &=& A(\partial (Q\backslash \Omega_{\mu'})) - 2V(Q\backslash \Omega_{\mu'}) \\[0.2cm]
&=& F(Q) - A(\partial Q \cap \Omega_{\mu'}) + A(Q\cap \partial \Omega_{\mu'}) + 2 V(Q\cap \Omega_{\mu'})\\[0.2cm]
&\leq & F(Q).
\end{eqnarray*}
To prove $ii)$, we first observe that  $\partial Q \cap \partial T_{r'}$ has one dimensional Hausdorff measure zero, for some $r' \in (r/2,r]$, as in the proof of item $i)$. Let $\xi$ be the outward unit vector field along the parallel surfaces foliating $T_{r'}(\overline{M})$. Notice that, independently of the mean curvature sign with respect to $\xi$, we must have $\dive \xi \leq 2$ a.e. in $T_{r'}(\overline{M})$. Set $\Gamma = T_{r'}(\overline{M})\backslash Q$ and compute
\begin{eqnarray*}
2V(\Gamma) &\geq & \int_{\Gamma} \dive \xi\\[0.2cm]
&=& \int_{\partial Q\cap T_{r'}(\overline{M})} \langle \xi, \eta(\Gamma)\rangle + \int_{\partial T_{r'}\backslash Q} \langle \xi, \eta(\Gamma)\rangle \\[0.2cm]
&\geq & - A(\partial Q\cap T_{r'}(\overline{M})) + A(\partial T_{r'}\backslash Q),
\end{eqnarray*}
where in the last inequality we have used that $\xi = \eta(\Gamma)$ along $\partial T_{r'}\backslash Q$. Thus, defining $Q' = Q \cup \Gamma$ it is easy to see that $Q' \in \mathcal{F}$ and
\begin{eqnarray*}
F(Q') &=& A(\partial Q) - A(\partial Q\cap T_{r'}(\overline{M})) + A(\partial T_{r'}\backslash Q) - 2V(Q) - 2V(\Gamma) \\[0.2cm]
&\leq & F(Q). 
\end{eqnarray*}

Once the above claim is proved, we can consider a minimizing sequence of subsets $Q_j \in \mathcal{F}$ satisfying $Q_j \cap \Omega_{\mu/2}=\emptyset$ and $T_{r'}(\overline{M}) \subset Q_j$. Thus, via the compactness theorem for integral currents (cf. \cite[Thm.5.5]{Morgan}) it is guaranteed the existence of a cluster point $Q_\infty$ as the limit of $Q_j$ in the flat topology, which still belongs to $\mathcal{F}$ and satisfies $Q_\infty \cap \overline{\Omega}_{\mu/2}=\emptyset$ and $T_{r'}(\overline{M}) \subset Q_\infty$. Since the area functional $A(\partial Q)$ is lower semi-continuous for the flat convergence and by the regularity properties of the volume differential form we can see that $Q_\infty$ minimizes $F$. Therefore, the piece of $\partial Q_\infty$ contained in the interior of $D_R$, here denoted by $S_R$, must be a local isoperimetric surface which by regularity theory (see \cite[Cor.3.6]{Morgan16}) is a smooth $1$-surface with mean curvature vector field pointing inward $Q_\infty$. Moreover, $S_R$ is strongly stable by \cite[Prop.2.3]{BC}, and the boundary of $S_R$ lies in the boundary of $B_R$.

For each positive integer $k$, let $R_k$ be a divergence sequence of radius, and let $S_k$ be the corresponding strongly stable $1$-surfaces constructed as above, that is, $S_k$ lies in the domain $D_k = D_{R_k}\backslash \overline{\Omega}_{\mu_{k}/2}$ and its boundary $\partial S_k$ belongs to $\partial B_{R_k}$. If we pick points $p \in \overline{M}$ and $q \in \partial \Omega$ such that the interior of the geodesic segment $[p,q]$ does not meet $\overline{M}$, then by construction, for $k$ sufficiently large, all the surfaces $S_k$ will intersect this segment. As in \cite{M}, for every $n\geq k$ the sequence $S_n$ satisfies a uniform local area estimate and has second fundamental form bounded in $D_k$, so this sequence admits a subsequence that converges, with multiplicity one, to an embedded strongly stable $1$-surface contained in $D_k$ and whose intersection with the geodesic segment $[p,q]$ is non-empty. Proceeding in a diagonal process we will find a smooth limit $S$ which is a complete $1$-surface strongly stable. Therefore, by \cite[Thm.2.13]{MPR} $S$ must be a horosphere.  
\end{proof}

\subsection{Proof of Theorem \ref{half_bc}}

We argue by contradiction and assume that $M$ lies in an open mean convex component $\Omega$ of $\mathbb{H}^3\backslash N$.  Hence, $N' \doteq \partial \Omega$ is given by a union of regular pieces of $N$ with mean curvature one with respect to normal vector field pointing towards the interior of $\Omega$, glued by their boundaries with inner angle less than or equal to $\pi$. If $N'\cap \overline{M} \not= \emptyset$, then by Lemma \ref{horosphere_lim_set} $N'$, thus $N$, must be a horosphere, which contradicts the non-horospherical assumption on $N$. On the other hand, if $N'\cap \overline{M} = \emptyset$, then Lemma \ref{horosphere_lim_set} gives the existence of a horosphere $S$ separating $\overline{M}$ from $N'$. Since $N$ is proper and the mean curvature vector field along $N'$ points towards $S$, we can apply \cite[Thm.7]{M} to conclude that $N$ is also a horosphere. This contradiction concludes the proof.

\section{Proof of Theorems \ref{parabolic_half}, \ref{stochastic_half-1-surfaces} and \ref{maximum-principle-infinity}}

Since the proofs to be presented in this section share a core argument based on the construction of a weak solution to the subequation $\triangle u \geq \lambda u$, for $\lambda\geq 0$, we will first provide a detailed introduction. Generically, let $\varphi \colon M \rightarrow P$ be an immersed surface on a manifold $P$ with Ricci curvature bounded from below and with sectional curvature bounded from above, and let $\psi \colon N \rightarrow P$ be a surface properly immersed in $P$ with bounded curvature. Henceforth, we assume that $M$ lies in a mean convex component of $P\backslash N$.

Following the approach from \cite{bessa-jorge-pessoa} a key point in the proof will be to consider solutions of the subequation $\Delta u \geq \lambda u$ in a weak sense, more precisely, in the barrier sense. We then recall that a continuous function $u \colon M \rightarrow \mathbb{R}$ is said to satisfy $\triangle u \geq 0$ at a point $p \in M$ in the barrier sense if, for any $\delta > 0$ there exists a smooth support function $\phi_\delta$ defined around $p$ such that
\begin{eqnarray*}
\begin{array}{cc}
\left\lbrace
\begin{array}{rl}
\phi_\delta =  u & \text{at} \ \ p, \\[0.2cm]
\phi_\delta  \leq u & \text{near } p,
\end{array}
\right.
& \quad \text{and} \qquad \triangle  \phi_\delta (p) > -\delta .
\end{array}
\end{eqnarray*}

The function $u$ to be constructed will be given in terms of the composition $u = g\circ t_{_{\!N}}\circ\varphi$ where $t_{_{\!N}} \colon N \rightarrow \mathbb{R}$ is the distance function to the surface $N$, and $g \colon \mathbb{R} \rightarrow \mathbb{R}$ is a smooth function to be chosen later satisfying $g'(t) < 0$ and $g''(t) >0$.


Let $\Omega$ be the mean convex connected component of $P\backslash N$ containing $M$, and let $\overrightarrow{H}_{\!_N} = \nu_{_N}$ be the mean curvature vector field along $\partial \Omega \subset N$ pointing towards $\Omega$. The boundary of $\Omega$ is given as a union of smooth pieces of $N$ and whose inner angles are not bigger than $\pi$ along an intersection set $\Gamma$.

From the curvature bounds of $P$ and $N$, we know that $N$ has the second fundamental form uniformly bounded. By the extended Rauch's theorem (see \cite[Cor.4.2]{warner}) there exists a regular tubular neighborhood $V(\varepsilon)$ of $N$, with $\varepsilon>0$ depending only on the curvature bounds, for which there is no focal points along any normal geodesic $\gamma : [0,\varepsilon) \to P$ issuing from a point $\gamma(0) \in N$ (cf. also \cite[Ch.10]{docarmo}). Therefore, along any geodesic minimizing the distance between a fixed point in $V(\varepsilon)$ and the surface $N$, the parallel surfaces along this normal geodesic are well-defined and non-degenerated.

Although the neighborhood $V(\varepsilon)$ may not be embedded, the distance function from $N$ restricted to $V_+(\varepsilon) = V(\varepsilon)\cap \Omega$, namely $t_{\!_N} : V_+(\varepsilon) \rightarrow \mathbb{R}$ is a positive Lipschitz function. Moreover, for a fixed point $y \in V_{+}(\varepsilon)$ it is easy to see that the nearest points to $y$ on $N$ can not lie on the part of $\Gamma$ where the inner angle is less than to $\pi$, for if a minimizing segment connecting $y$ to $\partial \Omega$ would be normal to two different tangent planes. 

Fix a point $y \in V_+(\varepsilon)$ and let $z \in N$ be a nearest point to $y$ and consider a simply connected, locally embedded neighborhood $W_z\subset N$ of $z$  that is  graph over an open ball  $B_z\subset T_{z}N$ with radius uniformly bounded from below, and such that ${\rm dist}_{\!_{P}}(y,z)\leq {\rm dist}_{\!_{P}}(y,\bar z)$ for all $\bar z \in W_z$. For 
each neighborhood $W_z$ the oriented distance function to $W_z$, namely $t_z  \colon C_z(\varepsilon) \to \mathbb{R}$, defined on a regular tubular neighborhood $C_z(\epsilon)=T_{\epsilon}(W_z)$ with radius $\epsilon$,  such that $C_z(\varepsilon)=C_z^{+}(\varepsilon)\cup W_z\cup C_z^{-}(\varepsilon)$ and $t_z(y) > 0$. Now, since $y$ could also be in the cut locus of $W_z$, to construct a smooth support function for $t_{\!_N}$, following notations from \cite{GLM} we consider a supporting surface $S_z$ for $C_z^+(\varepsilon)$ at $z \in W_z$, that is, a smooth surface such that $z \in S_z$ and $C_z^+(\varepsilon)\cap S_z = \emptyset$. Indeed, by \cite[Lem.1]{GLM}, for any $\mu>0$, there exists a supporting surface $S_z^{\mu}$ for $C_z^+(\varepsilon)$ at $z \in W_z$ such that
\begin{eqnarray}
H_{z}^{\mu}(z) > 1-\mu & \text{and} & y \notin \text{cut}(S_z^{\mu}),
\end{eqnarray}
where $H_{z}^{\mu}$ is the mean curvature of $S_z^{\mu}$. We notice that a way to construct these supporting surfaces is by deforming smoothly the boundary of a small ball $B \subset C_z^-(\varepsilon)$ touching $W_z$ at $z$. Furthermore, recalling $N$ has bounded curvature we can take a universal constant $0<c = \sup\{\vert\kappa^1\vert,\vert\kappa^2\vert\}< + \infty$, where $\kappa^1 \leq \kappa^2$ are the ordered principal curvatures of $S_z^{\mu}$, for all $z \in N$ and $\mu>0$ sufficiently small. Therefore, the oriented distance function to $S_z^{\mu}$, here called $t_{z}^{\mu}$, is smooth around $y$ and touches $t_{\!_N}$ from above at $y$.

Assuming that $\varphi(M) \cap V(\varepsilon/8) \not=\emptyset$, we can now define a bounded function $u \colon M \rightarrow \mathbb{R}$ by setting $u = \max\{v,0\}$, where $v \colon \varphi^{-1}(V_+(\varepsilon)) \rightarrow \mathbb{R}$ is given by $v(x) = g\circ t_{\!_N}(\varphi(x))$, with 
\begin{eqnarray}
g(t) = \log \left(\frac{2+\varepsilon\, c}{2+ 4\,c \,t}\right).
\end{eqnarray}


We first observe that $v(x) > 0$ if and only if $ x \in \varphi^{-1}(V_+(\varepsilon/4))$. Thus, $u$ will be a weak solution of $\Delta u \geq \lambda u$ on $M$, for $\lambda\geq 0$, once we have proved that $v$ satisfies this subequation on $\varphi^{-1}(V_+(\varepsilon/2))$ in the barrier sense. Up to reducing $\varepsilon$, if necessary, we will assume
\begin{eqnarray}\label{varepsilon_restrictions}
 0<\varepsilon < \tanh^{-1}\left(\frac{1}{4c}\right) < \frac{1}{2c} \raisepunct{.}
\end{eqnarray} 

For any fixed point $ p \in M$ with $y = \varphi(p) \in V_+(\varepsilon/2)$, we pick a point $z \in W_z \subset N$ and a neighborhood $V_z$ as described above. Given $\delta>0$, let us consider $\phi_\delta = g\circ t_{z}^{\mu}\circ \varphi$ as a support function to $v$ at $p$, where $t_{z}^{\mu}$ is the oriented distance function to $S_z^{\mu}$ with $t_{z}^{\mu}(y)>0$, and $S_z^{\mu}$ is the supporting surface provided by \cite[Lem.1]{GLM}, for some $\mu = \mu(\delta)>0$ to be chosen later. Since $t_z^{\mu}$ is smooth around $y$ and touches $t_{\!_N}$ from above, by the decreasing property of $g$, we can assert that $\phi_\delta$ is a smooth support function that touch $v$ from below at $p$. It is left to prove that
$$\triangle_{_M} \phi_\delta (p) > -\delta.$$  

To compute $\triangle_{_M} \phi_\delta$ we recall that 
\begin{eqnarray}\label{lapla_u}
\triangle_{_M} \phi_\delta &=& \text{Tr}_{TM} \text{Hess}_{P}(g\circ t_z^{\mu}) + 2\langle \nabla_{P}(g\circ t_z^{\mu}),\overrightarrow H_{\!_M}\rangle \nonumber\\[0.2cm]
&\geq & \text{Tr}_{TM}\left(g''(t_z^{\mu}) \nabla t_{z}^{\mu} \otimes \nabla t_{z}^{\mu} + g'(t_z^{\mu}) \nabla^2 t_{z}^{\mu}, \right) + 2g'(t_z^{\mu})\vert \overrightarrow H_{\!_M}\vert,
\end{eqnarray}
where 
\begin{eqnarray*}
g'(t) = - \frac{2c}{1+2c\, t}<0 & \text{and} & g''(t) = \frac{4c^2}{(1+2c\, t)^2}\raisepunct{.}
\end{eqnarray*}
The eigenvalues of $\text{Hess}_{P}(g\circ t_z^{\mu})$ are given by
$$ \mu_1 =  \frac{2c}{1+2c\, t_{z}^{\mu}}\kappa_{1}^{t} , \quad \mu_2 = \frac{2c}{1+2c\, t_{z}^{\mu}}\kappa_{2}^{t}, \qquad \text{and} \quad \mu_3 =  \frac{4c^2}{(1+2c\, t_{z}^{\mu})^2}\raisepunct{,}$$
where $\kappa_1^t, \kappa_2^t$ are the principal curvatures of the parallel surfaces to $S_z^{\mu}$ at $y$, and $\kappa_1 \leq \kappa_2$ are the principal curvatures of $S_z^{\mu}$. 

The main tool to estimate from below \eqref{lapla_u} is to consider the comparison theorem for the Riccati equation satisfied by the principal curvature of the parallel surfaces. For a fixed point $z \in N$ let $\xi$ be a unit-speed geodesic normal to $S_z^{\mu}$ at $z$ with $\xi(0) = z$, and let $\{\xi_1,\xi_2\}$ be an orthonormal basis that diagonalizes the Weingarten map on $T_z S_z^{\mu}$. Let us also denote by $H(t)$ be the signed mean curvature function of the parallel surfaces satisfying $\overrightarrow{H}_{S_t} = H(t) \xi'(t)$.  In the following lemma we summarize Corollaries 3.5 and 3.6 from \cite{gray2012tubes}.

\begin{lemma}\label{gray_lemmata}
Let $\xi_1, \xi_2$ be two vector fields differentiable at time $t$. Then,
\begin{enumerate}
\item[a)] $ \kappa'_i(t) = \kappa_i^2(t) + \text{Sec}_{\!_{P}}(\xi'(t),\xi_i(t)).$ 
\vspace{0.2cm}
\item[b)] $ 2H'(t) = \kappa_1^2(t) + \kappa_2^2(t) + \text{Ric}_{\!_{P}}(\xi'(t)).$
\end{enumerate}
\end{lemma}

With this preliminaries in hand we can now proceed with the proof of each theorem independently.

\subsection{Proof of Theorem \ref{parabolic_half}}
The proof follows by contradiction, so we assume that $N$ is a proper $1$-surface immersed in $\mathbb{H}^3$ with bounded curvature and that $M$ is a parabolic $1$-surface contained in a mean convex component of $\mathbb{H}^3\backslash N$ which is not parallel to $N$. As before $\Omega$ denotes this mean convex component and  $\overrightarrow{H}_{\!_N} = \nu_{_N}$ is the mean curvature vector field along $\partial \Omega \subset N$ pointing towards $\Omega$. 
Moreover, since we are working on the hyperbolic space $\mathbb{H}^3$ up to an isometry we may assume $M \cap V(\varepsilon/8) \not= \emptyset$.

From hypotheses the function $u$ described above must be non-constant, thus the parabolicity of $M$ will get a contradiction if we prove that $\triangle u \geq 0$ on $M$ in a weak sense. Indeed, applying \cite[Thm.5.1]{grigoryan} we have that $u$ must be constant.

We first recall that the principal curvatures $\kappa_i^t$, for $i=1,2$, satisfying the equation $a)$ in Lemma \ref{gray_lemmata} are solutions of the following Riccati equation
\begin{eqnarray*}
\left(\kappa_i^t\right)' = \left(\kappa_i^t\right)^2 - 1,
\end{eqnarray*}
along orthogonal geodesics issuing from a point in $N$. These solutions are explicitly given by
\begin{eqnarray*}
\kappa_i^t = \frac{\kappa_i \cosh t - \sinh t}{\cosh t - \kappa_i \sinh t}\raisepunct{.}
\end{eqnarray*}

The restrictions on the value of $\varepsilon$ imposed in \eqref{varepsilon_restrictions} allow us to have a good monotonicity for the eigenvalues 
\begin{eqnarray}\label{monotonicity_ineq}
\mu_1 \leq \mu_2 < \mu_3.
\end{eqnarray}
Indeed, the former inequality follows from the monotonicity $\kappa_1 \leq \kappa_2$. For the last one, note that since $0<t< \varepsilon/2 < 1/4c$ we easily deduce
\begin{eqnarray}\label{eq_aux_1}
\frac{2c}{1+2ct} > \frac{4c}{3}\raisepunct{.}
\end{eqnarray} 
Using $\tanh t \leq 1/4c$ we also obtain
\begin{eqnarray*}
\tanh t \leq \frac{c}{4c^2 - 3}\raisepunct{,}
\end{eqnarray*}
which turns out to be equivalent to
\begin{eqnarray}\label{eq_aux_2}
\frac{c\,\cosh t - \sinh t}{\cosh t - c\, \sinh t} \leq \frac{4c}{3}\raisepunct{.}
\end{eqnarray}
Putting together \eqref{eq_aux_1} with \eqref{eq_aux_2}, and recalling that $\vert \kappa_i\vert \leq c$, we finally have
\begin{eqnarray*}
k_2^t = \frac{\kappa_2 \cosh t - \sinh t}{\cosh t - \kappa_2 \sinh t} \leq \frac{c\,\cosh t - \sinh t}{\cosh t - c\, \sinh t} < \frac{2c}{1+2ct}\raisepunct{.}
\end{eqnarray*}
Thus, the inequality $\mu_2^t < \mu_3^t$ follows.

Now, applying  \cite[Lem.2.3]{jorge2003barrier} we can estimate \eqref{lapla_u} from below as
\begin{eqnarray}\label{lapla_u_ineq}
\triangle \phi_\delta &\geq & -g'(t_z^{\mu})\left(\kappa_1^t + \kappa_2^t\right) + 2g'(t_z^{\mu}) \nonumber \\[0.2cm]
&=& -2g'(t_z^{\mu})\left(H(t) - 1\right).
\end{eqnarray}

To estimate \eqref{lapla_u_ineq} we first note that due to Newton's inequality it is easy to deduce that 
\begin{eqnarray*}
\left\lbrace 
\begin{array}{l}
H'(t)  \geq  H^2(t) - 1, \\[0.2cm]
H(0)  =  H \geq 1-\mu ,
\end{array}\right.
\end{eqnarray*}
where $H$ is the mean curvature of $S_z^\mu$. By the Riccati's comparison theorem we may have
\begin{eqnarray}\label{H_ineq_final}
H(t) \geq  \frac{H \cosh t - \sinh t}{\cosh t - H \sinh t}\raisepunct{.}
\end{eqnarray}
Substituting \eqref{H_ineq_final} into \eqref{lapla_u_ineq} and recalling $H>0$ and  $0< t < \tanh^{-1}(1/4c)$ we obtain
\begin{eqnarray*}
\triangle \phi_\delta &\geq & -2g'(t_z^{\mu}) \frac{(H-1)(\cosh t + \sinh t)}{\cosh t - H \sinh t} \\[0.2cm]
&\geq & 2g'(t_z^{\mu}) \frac{\cosh t + \sinh t}{\cosh t - H \sinh t}\mu \\[0.2cm]
&\geq & -\frac{4c}{1+2ct} \frac{1 + \tanh t}{1 - 2c \tanh t}\mu \\[0.2cm]
&\geq & -16 c\,\mu .
\end{eqnarray*}

Now, taking $\mu(\delta) = \delta/16c$ we conclude that $\triangle \phi_\delta \geq -\delta$, and the function $v$ satisfies $\triangle v \geq 0$ in the barrier sense on $\varphi^{-1}(V_+(\varepsilon/2))$. Therefore, the function $u = \max\{v,0\}$ must be constant. A contradiction. 

\subsection{Proof of Theorem \ref{stochastic_half-1-surfaces}}

To prove the stochastic theorem we need to obtain a strong inequality for the Laplacian of the function $u$, namely, we should prove that $\triangle u \geq \lambda u$, for some $\lambda>0$, in the barrier sense. Following computations from the proof of Theorem \ref{parabolic_half} it is easy to see that \eqref{monotonicity_ineq} holds. Thus, substituting in \eqref{lapla_u} and using \eqref{H_ineq_final} we have 
\begin{eqnarray}\label{ineq_phi_stochastic}
\triangle \phi_\delta &\geq & -2g'(t_z^{\mu})\left(H(t) - H_{\!_M}\right) \nonumber\\[0.2cm]
&\geq & -2g'(t_z^{\mu})\left(\frac{H \cosh t - \sinh t}{\cosh t - H \sinh t} - H + (H - H_{\!_M})\right) \nonumber \\[0.2cm]
&\geq & -2g'(t_z^{\mu})\left(\frac{(H^2 - 1)\sinh t}{\cosh t - H \sinh t} + (H_{\!_N} - H_{\!_M}) - \mu \right) .
\end{eqnarray}

To prove item $i)$ we use that $H\geq H_{\!_N} \geq 1-\mu$ and $0< t < \tanh^{-1}(1/4c)$ to compute
\begin{eqnarray*}
\triangle \phi_\delta &\geq & -2g'(t_z^{\mu})\left(\frac{(\mu^2 - 2\mu)\sinh t}{\cosh t - H \sinh t} + (H_{\!_N} - H_{\!_M}) - \mu \right) \\[0.2cm]
&\geq & -2g'(t_z^{\mu})\left(\frac{- 2\mu\tanh t}{1 - 2c \tanh t} + (H_{\!_N} - H_{\!_M}) - \mu \right) \\[0.2cm]
&\geq & \frac{4c}{1+2c\,t_z^{\mu}}(H_{\!_N} - H_{\!_M}) -  \frac{20c}{1+2c\,t_z^{\mu}}\mu  \\[0.2cm]
&\geq & (\inf_{N} H_{\!_N} - \sup_{M} H_{\!_M})\phi_\delta - 20c\mu .
\end{eqnarray*}
Taking $\mu(\delta) = \delta/20c > 0$ and $\lambda = (\inf_{N} H_{\!_N} - \sup_{M} H_{\!_M})>0$ we conclude that $\triangle v \geq \lambda v$ holds in the barrier sense. Again, from the Liouville theorem for stochastically complete surfaces \cite[Thm.6.1]{grigoryan} we must conclude that the function $u = \max\{v,0\}$ must be identically zero. This finishes the proof.

Similarly, to prove item $ii)$ we assume by contradiction the existence of a positive constant $\ell$ with $\ell \leq t \leq \varepsilon/8$. Taking $\mu(\delta)$ sufficiently small such that $H^2 - 1 > (\inf_N H_{\!_N}^2 - 1)/2 > 0$, by \eqref{ineq_phi_stochastic} we obtain
\begin{eqnarray*}
\triangle \phi_\delta &\geq & -2g'(t_z^{\mu})\left(\frac{(H^2 - 1)}{\coth t - 1} + (H_{\!_N} - H_{\!_M}) - \mu \right) \\[0.2cm]
&\geq & \frac{4c}{1+2c\,t_z^{\mu}}\left(\frac{\inf_N H_{\!_N}^2 - 1}{2\coth \ell - 2} + \inf_{N} H_{\!_N} - \sup_{M} H_{\!_M} - \mu \right) \\[0.2cm]
&\geq & \frac{\inf_N H_{\!_N}^2 - 1}{2\coth \ell - 2}\phi_\delta - 4c\mu .
\end{eqnarray*}
Thus, taking $\mu(\delta) = \delta/4c$ and $\lambda = (\inf_N H_{\!_N}^2 - 1)/(2\coth \ell - 2)>0$ we will arrive at the same contradiction as above.

\subsection{Proof of Theorem \ref{maximum-principle-infinity}}

As in the proof of Theorems \ref{parabolic_half} and \ref{stochastic_half-1-surfaces}, from a translation argument we can assume that ${\rm dist}(M,N) = 0$. Moreover, the selected function $v = g\circ t_{\!_N}$ used before will be a bounded solution of $\triangle v \geq 0$ on $\varphi^{-1}(U(\varepsilon/2)) \cap {\rm int} M$, where $\varphi$ denotes the isometric immersion of $(M, \partial M)$ into $\mathbb{H}^3_{\raisepunct{.}}$ Therefore, $u = \max\{v,0\}\in C^0(M)\cap W^{1,2}_{loc}({\rm int} M)$ is a weak bounded  subharmonic function on ${\rm int} M$ and since $M$ is parabolic, the Ahlfors maximum principle \cite[Prop.10]{pessoa-pigola-setti} says that 
\begin{eqnarray*}
\sup_M u = \sup_{\partial M} u.
\end{eqnarray*}
The conclusion then follows by noticing that $u(x) \to \sup_M u$ if and only if ${\rm dist} (\varphi(x),N) \to 0$.

%



{}
 
\end{document}